\documentclass{amsart}
\usepackage{enumerate,verbatim}
\usepackage[all,2cell,ps]{xy}
\usepackage{mathptmx}
\usepackage[notcite, notref]{}
\usepackage[pagebackref]{hyperref}

\usepackage{nicefrac}
\usepackage{array}
\usepackage{xcolor}

\usepackage{amsthm}
\usepackage[leqno]{amsmath}
\usepackage[pagebackref]{hyperref}

\usepackage{mathptmx}

\usepackage{amsfonts,amsmath,amssymb,amsthm,amscd,amsxtra}
\usepackage{enumerate,verbatim}
\usepackage[usenames,dvipsnames]{pstricks}
\usepackage[mathscr]{eucal}
\usepackage{amsfonts,amsmath,amssymb,amsthm,amscd,amsxtra}
\usepackage{amssymb,amstext,amsmath,amscd,amsthm,amsfonts,enumerate,graphicx,latexsym}

\usepackage{enumerate,verbatim}
\usepackage[all,2cell,ps]{xy}
\usepackage[notcite, notref]{}
\usepackage[pagebackref]{hyperref}
\usepackage{todonotes}

\usepackage[all,2cell,ps]{xy}
\usepackage[notcite,notref]{}
\usepackage[pagebackref]{hyperref}

\DeclareMathOperator{\md}{\operatorname{\mathsf{mod}}}

\DeclareMathOperator{\nf}{\operatorname{\mathsf{NF}}}

\theoremstyle{plain}

\newtheorem{thm}{Theorem}[section]

\newtheorem{cor}[thm]{Corollary}

\theoremstyle{definition}

\newtheorem{eg}[thm]{Example}
\newtheorem{conj}[thm]{Conjecture}

\numberwithin{equation}{section}

\newcommand{\fm}{\mathfrak{m}}

\newcommand{\fp}{\mathfrak{p}}

\newcommand{\QQ}{\mathbb{Q}}

\newtheorem{chunk}[thm]{\hspace*{-1.065ex}\bf}

\def\G-dim{\operatorname{\mathsf{G-dim}}}

\def\pd{\operatorname{\mathsf{pd}}}

\def\Tr{\mathsf{Tr}\hspace{0.01in}}

\DeclareMathOperator{\coker}{coker}

\def\depth{\operatorname{\mathsf{depth}}}

\def\Ext{\operatorname{\mathsf{Ext}}}

\DeclareMathOperator{\height}{height}

\def\Hom{\operatorname{\mathsf{Hom}}}

\DeclareMathOperator{\Supp}{Supp}
\DeclareMathOperator{\Spec}{Spec}

\def\Tor{\operatorname{\mathsf{Tor}}}

\def\urltilda{\kern -.15em\lower .7ex\hbox{\~{}}\kern .04em}
\def\urldot{\kern -.10em.\kern -.10em}\def\urlhttp{http\kern -.10em\lower -.1ex
\hbox{:}\kern -.12em\lower 0ex\hbox{/}\kern -.18em\lower 0ex\hbox{/}}

\begin{document}
\title[On the vanishing of self extensions over Cohen-Macaulay local rings]{On the vanishing of self extensions over \\Cohen-Macaulay local rings}



\thanks{Araya and Takahashi were partly supported by JSPS Grants-in-Aid for
Scientific Research 26400056 and 16K05098, respectively. Sadeghi's research was supported by a grant from IPM}

\author[T. Araya]{Tokuji Araya}
\address{Tokuji Araya\\
Department of Applied Science, Faculty of Science, Okayama University of Science, Ridaicho, Kitaku, Okayama 700-0005, Japan.}
\email{araya@das.ous.ac.jp}

\author[O. Celikbas]{Olgur Celikbas}
\address{Olgur Celikbas\\
Department of Mathematics \\
West Virginia University\\
Morgantown, WV 26506-6310, U.S.A}
\email{olgur.celikbas@math.wvu.edu}

\author[A. Sadeghi]{Arash Sadeghi}
\address{Arash Sadeghi\\
School of Mathematics, Institute for Research in Fundamental Sciences (IPM), P.O. Box: 19395-5746, Tehran, Iran}
\email{sadeghiarash61@gmail.com}

\author[R. Takahashi]{Ryo Takahashi}
\address{Ryo Takahashi\\ Graduate School of Mathematics, Nagoya University, Furocho, Chikusaku, Nagoya 464-8602, Japan}
\email{takahashi@math.nagoya-u.ac.jp}
\urladdr{http://www.math.nagoya-u.ac.jp/~takahashi/}

\setcounter{tocdepth}{1}

\begin{abstract} The celebrated Auslander-Reiten Conjecture, on the vanishing of self extensions of a module, is one of the long-standing conjectures in ring theory. Although it is still open, there are several results in the literature that establish the conjecture over Gorenstein rings under certain conditions. The purpose of this article is to obtain extensions of such results over Cohen-Macaulay local rings that admit canonical modules. In particular, our main result recovers theorems of Araya, and
Ono and Yoshino simultaneously.
\end{abstract}

\maketitle{}

\section{Introduction}
Throughout $R$ denotes a commutative Noetherian local ring and $\md R$ denotes the category of all finitely generated $R$-modules. 

There are various conjectures from the representation theory of algebras that have been transplanted to Commutative Algebra. One of the most important such conjectures is the celebrated \emph{Auslander-Reiten Conjecture} \cite{AuRe}, which states that a finitely generated module $M$ over an Artin algebra $A$ satisfying $\Ext^i_A(M,M) = \Ext^i_A(M,A) =0$ for all $i\geq 1$ must be projective. This long-standing conjecture is closely related to other important conjectures such as the \emph{Nakayama's Conjecture} \cite{Nak} and the \emph{Tachikawa Conjecture} \cite{ABS, Tach}. Although the Auslander-Reiten Conjecture was initially proposed over Artin algebras, it can be stated over arbitrary Noetherian rings; in local algebra, the conjecture is known as follows:

\begin{conj} (Auslander and Reiten \cite{AuRe}) \label{ARconj} Let $R$ be a local ring and let $M\in \md R$. If $\Ext^i_R(M,M) = \Ext^i_R(M,R) =0$ for all $i\geq 1$, then $M$ is free.
\end{conj}

Recently there have been significant interest and hence some progress towards the Auslander-Reiten conjecture; see, for example, \cite{CT, CH1, CH2, GT, HL, CSV}. A particular result worth recording on the Auslander-Reiten Conjecture is due to Huneke and Leuschke: the Auslander-Reiten Conjecture holds over Gorenstein normal domains; see \cite[1.3]{HL}. Another result in this direction, which is of interest to us, is due to Araya \cite{Ar}. For a nonnegative integer $n$ and an $R$-module $M$, we set $X^{n}(R)=\{\fp \in \Spec(R): \height(\fp) \leq n\}$ and say $M$ is \emph{locally free on} $X^{n}(R)$ if $M_{\fp}$ is a free $R_{\fp}$-module for each prime ideal $\fp \in X^{n}(R)$.

\begin{thm} (Araya \cite{Ar}) \label{Tokuji} Let $R$ be a Gorenstein local ring of dimension $d\geq 2$ and let $M\in \md R$. Then $M$ is free provided that the following hold:
\begin{enumerate}[(i)]
\item $M$ is locally free on $X^{d-1}(R)$.
\item $M$ is maximal Cohen-Macaulay.
\item $\Ext^{d-1}_R(M, M)=0$.
\end{enumerate}
\end{thm}

Ono and Yoshino \cite{OnYos} relaxed the condition that $M$ is free on $X^{d-1}(R)$ in Araya's theorem under the hypothesis that $\Ext^i_R(M,M)$ vanishes for $i=d-2$ and $i=d-1$. More precisely they proved:

\begin{thm} (Ono and Yoshino \cite{OnYos}) \label{OY}
Let $R$ be a Gorenstein local ring of dimension $d\geq 3$ and let $M\in \md R$ be a module. Then $M$ is free provided that the following hold:
\begin{enumerate}[(i)]
\item $M$ is locally free on $X^{d-2}(R)$.
\item $M$ is maximal Cohen-Macaulay.
\item $\Ext^{d-2}_R(M, M)=\Ext^{d-1}_R(M, M)=0$.
\end{enumerate}
\end{thm}

We set $(-)^{\ast}=\Hom_R(-,R)$, $(-)^{\vee}=\Hom(-, E)$, where $E$ is the injective hull of the residue field of $R$, and $(-)^{\dagger}=\Hom_R(-,\omega)$ where $\omega$ is a canonical module of $R$.
Recall that $M\in \md R$ is said to satisfy $(S_2)$ if $\depth_{R_{\fp}}(M_{\fp})\geq \{2, \height_R(\fp)\}$ for all $\fp \in \Supp(M)$.

The main aim of this paper is to prove the following result.

\begin{thm}\label{t1}
	Let $R$ be a Cohen-Macaulay local ring of dimension $d$ with
	canonical module $\omega$ and let $M\in \md R$. Assume $n$ is an integer with $1\le n\leq d-1$. Then $M$ is free provided that the following hold:
	\begin{enumerate}[\rm(i)]
		\item $\pd_{R_{\fp}}(M_{\fp})<\infty$ for all $\fp \in X^{n}(R)$.
		\item $M$ satisfies $(S_2)$ and $M^{\ast}$ is maximal Cohen-Macaulay.
		\item $\Ext^{i}_R(M,(M^{\ast})^{\dagger})=0$ for all $i=n, \ldots, d-1$.
	\end{enumerate}
\end{thm}

An immediate consequence of Theorem \ref{t1} is:

\begin{cor}\label{cormainCM}Let $R$ be a Cohen-Macaulay normal local domain of dimension $d$ with
canonical module $\omega$ and let $M\in \md R$. Then $M$ is free provided that the following hold:
\begin{enumerate}[\rm(i)]
\item $M$ satisfies $(S_2)$ and $M^{\ast}$ is maximal Cohen-Macaulay.
\item $\Ext^{i}_R(M,(M^{\ast})^{\dagger})=0$ for all $i=1, \ldots, d-1$.
\end{enumerate}
\end{cor}

As maximal Cohen-Macaulay modules satisfy $(S_2)$ over Gorenstein rings, we deduce from Theorem \ref{t1} that:

\begin{cor}\label{cormain}
Let $R$ be a Gorenstein local ring of dimension $d$ and let $M\in \md R$. Assume $n$ is an integer with $1\le n\leq d-1$. Then $M$ is free provided that the following hold:
\begin{enumerate}[\rm(i)]
\item $M$ is locally free on $X^{n}(R)$.
\item $M$ is maximal Cohen-Macaulay.
\item $\Ext^{i}_R(M,M)=0$ for all $i=n, \ldots, d-1$.
\end{enumerate}
\end{cor}

Note that we recover heorems \ref{Tokuji} and \ref{OY} from Corollary \ref{cormain} by letting $n=d-1$ and $n=d-2$, respectively. Corollary 1.6 especially yields an extension of a result of Huneke and Leuschke \cite{HL} mentioned preceding Theorem \ref{Tokuji}. More precisely we obtain the following; see Corollary \ref{cor1}.

\begin{cor} \label{corintro} Let $R$ be a $d$-dimensional Gorenstein local normal domain and let $M$ be a maximal Cohen-Macaulay $R$-module. Then $M$ is free if and only if $\Ext^i_R(M,M)=0$ for all $i=1, \ldots, d-1$.
\end{cor}

Let us remark that Huneke and Leuschke \cite[3.1]{HL} obtains the conclusion of Corollary \ref{corintro}, when $\Ext^i_R(M,M)$ vanishes for all $i=1, \ldots, d$; see the discussion following Corollary \ref{cor1}. Let us also remark that one can in fact prove Theorem \ref{OY} and Corollary \ref{corintro} by using the proof of Theorem \ref{Tokuji} given by Araya; see \cite{Ar}. Our main argument is more general than both of these results; it is quite short and works over Cohen-Macaulay rings that are not necessarily Gorenstein. Hence we will deduce Theorem \ref{OY} and Corollary \ref{corintro} as immediate corollaries of Theorem \ref{t1} in the next section without making use of the proof of Theorem \ref{Tokuji}; see also Corollary \ref{c1}.

\section{Main result and corollaries}
In this section we give a proof of our main result, Theorem \ref{t1}. Following our proof we state two corollaries, one of which extends the result of Huneke and Leuschke \cite{HL} mentioned preceding Theorem \ref{Tokuji}. We start with a few notations and preliminary results.

\begin{chunk} \label{Triso} Let $M,N \in \md R$. We denote by $\underline{\Hom}_R(M,N)$ the residue of $\Hom_R(M,N)$ by the $R$-submodule consisting of the $R$-module homomorphisms from $M$ to $N$ that factor through free modules.

It follows from the definition that $M$ is free if and only if $\underline{\Hom}_R(M,M)=0$. We also remark that $\underline{\Hom}_R (M, N) \cong \Tor_1^R (\Tr M, N)$, where $\Tr M$ is the (Auslander) transpose of $M$; see \cite[3.9]{Yo}.
\end{chunk}

\begin{chunk} \label{Sn} (\cite[1.4.1]{BH}) Let $R$ be a Cohen-Macaulay local ring such that $\depth(R)\geq 2$ and let $M\in \md R$. If $M$ satisfies $(S_2)$ and $\pd_{R_{\fp}}(M_{\fp})<\infty$ for all $\fp \in X^1(R)$, then $M$ is reflexive.
\end{chunk}

\begin{chunk} (\cite[2.2]{YKI}) \label{Yoshida} Let $R$ be a local ring and let $M,N \in \md R$. If $N$ is maximal Cohen-Macaulay and $\pd(M)<\infty$, then $\Tor_i^R(M,N)=0$ for all $i\geq 1$.
\end{chunk}

For our proof of Theorem \ref{t1}, we will make use \cite[3.1]{CeD2}, which is well-known over Artinian rings. Since we record an improved version of that result, we give a proof along with its statement.

\begin{chunk} (\cite[3.1]{CeD2}) \label{CDCT} Let $R$ be a $d$-dimensional Cohen-Macaulay local ring with a canonical module $\omega$ and let $M,N \in \md R$. If $\pd_{R_{\fp}}(M_{\fp})<\infty$ for all $\fp \in X^{d-1}(R)$ and $N$ is maximal Cohen-Macaulay, then the following isomorphism holds for all $i\geq 1$:
$$ \Ext_R^{d+i}(M,N^{\dagger}) \cong \Ext_{R}^d(\Tor_i^R(M,N),\omega).$$

To see this isomorphism, we note, by \cite[10.62]{Roit}, that there is a third quadrant spectral sequence:
$$\Ext^p_{R}(\Tor_q^R(M,N),\omega) \underset{p}{\Longrightarrow} \Ext_R^n(M,N^{\dag}) $$
Let $p\in X^{d-1}(R)$. Then, since $\pd_{R_{\fp}}(M_{\fp})<\infty$ and $N_p$ is maximal Cohen-Macaulay over $R_p$, we conclude that $\Tor_i^R(M,N)_{p}=0$ for all $i\geq 1$; see \ref{Yoshida}.  Therefore
$\Tor_q^R(M,N)$ has finite length for all $q>0$. Hence, unless $p=d$, $\Ext^p_{R}(\Tor_q^R(M,N),\omega)=0$. It follows that the spectral sequence considered collapses and hence gives the desired isomorphism.
\end{chunk}

We can now prove our main result.

\begin{proof} [A proof of Theorem \ref{t1}] Note, it follows from our hypotheses and \ref{Sn}, that $M$ is reflexive. Set $N=M^{\ast}$ and note $M \cong \Omega^2 \Tr N$. Therefore $\pd_{R_{\fp}}((\Tr N)_{\fp})<\infty$ for all $\fp \in X^{n}(R)$.

Let $\fp \in X^{n}(R)$. Then, since $N_{\fp}$ is maximal Cohen-Macaulay, it follows from \ref{Yoshida} that $\Tor_i^R(\Tr N,N)_{\fp}=0$ for all $i\geq 1$. In particular $\Tor_1^R(\Tr N,N)_{\fp}=0$, and this implies $N_{\fp}$ is free over $R_{\fp}$; see \ref{Triso}. Since $M$ is reflexive, we conclude that $M_{\fp}$ is free. Consequently $M$ is a reflexive module that is locally free on $X^{n}(R)$.

Let $\fp \in \Spec(R)$. We proceed by induction on $\height_R(\fp)$ and prove that $M_{\fp}$ is free over $R_{\fp}$. If $\height(\fp)\leq n$, then $M_{\fp}$ is free by the above argument. So we assume $\height(\fp)=t>n$. Localizing at $\fp$, we may assume $(R, \fm)$ is a Cohen-Macaulay local ring with a canonical module $\omega$, $\dim(R)=t> n\geq 1$, $M$ is reflexive and locally free on $X^{t-1}(R)$, $M^{\ast}$ is maximal Cohen-Macaulay and $\Ext^{t-1}_R(M,(M^{\ast})^{\dagger})=0$.

Note that $N^{\ast}=M^{\ast\ast}\cong M$ and $N^{\dagger}=(M^{\ast})^{\dagger}$ is maximal Cohen-Macaulay. Moreover $\Ext^{t-1}_R(N^*, N^\dagger)=0$ by the hypothesis. Hence the following isomorphisms hold:
\[\begin{array}{rl}\tag{\ref{t1}.1}
0=\Ext^{t-1}_R(N^{\ast}, N^\dagger)^{\vee} &\cong\Ext^{t+1}_R(\Tr N, N^{\dagger})^{\vee}\\
&\cong\Ext^t_R(\Tor_1^R(\Tr N,N),\omega)^{\vee}\\
&\cong\Gamma_{\fm}(\Tor_1^R(\Tr N,N))\\
&\cong\Tor_1^R(\Tr N,N)\\
&\cong\underline{\Hom}_R(N,N).
\end{array}\]
The first isomorphism of (\ref{t1}.1) follows from the fact that $N^{\ast} \approx\Omega^2\Tr N$, while the second and third ones follow from \ref{CDCT} and the Local Duality Theorem \cite[3.5.9]{BH}, respectively. The fourth isomorphism is due to the fact that $\Tor_1^R(\Tr N,N)$ has finite length. The fifth isomorphism, and the freeness of $N$, follows from \ref{Triso}. As $M$ is reflexive, we conclude that $M \cong N^{\ast}$ is free.
\end{proof}

If $R$ is a Gorenstein local ring and $M$ is a maximal Cohen-Macaulay $R$-module, then $M$ satisfies $(S_2)$, $M^{\ast}$ is maximal Cohen-Macaulay and $(M^{\ast})^{\dagger} \cong M$. Hence, as an immediate consequence of Theorem \ref{t1}, we obtain Corollary \ref{cormain}. 

To obtain another corollary of Theorem \ref{t1}, we apply the next result:

\begin{chunk} (\cite[B4]{ABS}) \label{obs} Let $R$ be a Cohen-Macaulay local ring of dimension $d$ with a canonical module $\omega$ and let $M\in \md R$. If $\Ext^i_R(M,R)=0$ for all $i=1, \ldots, d$, then $M\otimes_R\omega \cong (M^{\ast})^{\dagger}$.
\end{chunk}

\begin{cor} \label{cor1}
Let $R$ be a Cohen-Macaulay local ring of dimension $d$ with a canonical module $\omega$ and let $M\in \md R$. Assume $n$ is an integer with $1\le n\leq d-1$. Then $M$ is free provided that the following holds:
\begin{enumerate}[\rm(i)]
\item $\pd_{R_{\fp}}(M_{\fp})<\infty$ for all $\fp \in X^{n}(R)$.
\item $M$ is reflexive and $\Ext^i_R(M,R)=0$ for all $i=1, \ldots, d$ (e.g., $M$ is totally reflexive.)
\item $\Ext^{i}_R(M,M\otimes_R\omega)=0$ for all $i=n, \ldots, d-1$.
\end{enumerate}
\end{cor}

\begin{proof} The vanishing of $\Ext^i_R(M,R)$ for all $i=1, \ldots, d$ forces $M^{\ast}$ to be a $d$th syzygy module, i.e., a maximal Cohen-Macaulay module. Therefore the conclusion follows from Theorem \ref{t1} and  \ref{obs}.
\end{proof}

We point out that the case where $R$ is Gorenstein and $n=1$ of Corollary \ref{cor1} yields Corollary \ref{corintro} advertised in the introduction.

Corollary \ref{cor1} allows us to generalize \cite[5.5]{OnYos}, another result of Ono and Yoshino.

\begin{cor}\label{c1}
Let $R$ be a Cohen-Macaulay local ring of dimension $d$ with a canonical module $\omega$ and let $M\in \md R$.
Assume $n$ is an integer with $1\leq n \leq d-1$. Then $M$ is free provided that the following holds:
\begin{enumerate}[(i)]
 \item $M$ is locally free on $X^{d-n}(R)$.
 \item $M$ is reflexive and $\Ext^i_R(M,R)=0$ for all $i=1, \ldots, d$.
 \item $\depth_R(\Ext^{d-i}_R(M,M\otimes_R\omega))\geq i$ for all $i=1, \ldots, n$.
\end{enumerate}
\end{cor}

\begin{proof} Suppose $M$ is not free. Then the nonfree locus $\nf(M)$ of $M$ is not empty. Let $\nf(M)=V(I)$ for some ideal $I$ of $R$ with $\height_R(I)=t$. Then it follows from (i) that $t\geq d-n+1$. Setting $i=d-t+1$, we obtain from (iii) that 
\begin{equation}\tag{\ref{c1}.1}
\depth_R(\Ext^{t-1}_R(M, M\otimes_R\omega))\geq d-t+1.
\end{equation}
On the other hand, since $X^{t-1}(R) \nsubseteq V(I)$, we see $M$ is locally free on $X^{t-1}(R)$. This implies
\begin{equation}\tag{\ref{c1}.2}
\dim_R(\Ext^{t-1}_R(M, M\otimes_R\omega))\leq d-t.
\end{equation}
Consequently, by (\ref{c1}.1) and (\ref{c1}.2), we conclude
$\Ext^{t-1}_R(M,M\otimes_R\omega)=0$. 

Let $\fp \in \nf(M)$ with $\height_R(\fp)=t$. Then $M_{\fp}$ is locally free on $X^{t-1}(R_{\fp})$, $M_{\fp}$ is reflexive over $R_{\fp}$, $\Ext^i_{R_{\fp}}(M_{\fp},R_{\fp})=0$ for all $i=1, \ldots, t$, and $\Ext^{t-1}_{R_\fp}(M_\fp,M_\fp \otimes_{R_{\fp}}\omega_{\fp})=0$. Therefore it follows from Corollary \ref{cor1} that $M_\fp$ is free over $R_{\fp}$, which is a contradiction. So $M$ is free.
\end{proof}

Huneke and Leuschke \cite[3.1]{HL} proved, when $R$ is a $d$-dimensional complete Cohen-Macaulay local ring such that $R_{\fp}$ is a complete intersection for all $\fp \in X^{1}(R)$, and $R$ is either Gorenstein or contains $\QQ$, a maximal Cohen-Macaulay $R$-module $M$ of constant rank is free provided $\Ext^i_R(M,R)=\Ext^i_R(M^{\ast},R)=0$ for all $i=1, \ldots, d$ and $\Ext^i_R(M,M)=0$ for all $i=1, \ldots, \max\{2, d\}$. Consequently, when $R$ is a Gorenstein normal domain of dimension $d\geq 2$ and $M$ is a maximal Cohen-Macaulay $R$-module, the result of Huneke and Leuschke \cite[3.1]{HL} requires the vanishing of $\Ext^i_R(M,M)=0$ for all $i=1, \ldots, d$ to conclude that $M$ is free, whilst Corollary \ref{corintro} requires the vanishing of $\Ext^i_R(M,M)=0$ for all $i=1, \ldots, d-1$ for the same conclusion under the same setup.

As Huneke and Leuschke \cite{HL} studied the Auslander and Reiten Conjecture for modules that have constant rank, it seems worth finishing this section with a related example: it shows that the hypothesis that $M$ is locally free on $X^{d-1}(R)$ cannot be removed from Corollary \ref{cormain} even if $M$ has constant rank.

\begin{eg} \label{egTokuji}
Let $k$ be a field and let $R:=k[\![x,y,z,u,v]\!]/(xy-uv)$. Then $R$ is a four dimensional hypersurface domain. In particular any module in $\md R$ has constant rank.

Consider the following minimal free resolution:

\begin{equation} \tag{\ref{egTokuji}.1}
F=(\cdots \overset{B}{\longrightarrow} R^2 \overset{A}{\longrightarrow} R^2 \overset{B}{\longrightarrow} R^2 \overset{A}{\longrightarrow} R^2 \to 0),
\end{equation}\\
where $A=\begin{pmatrix}
x & u \\
v & y
\end{pmatrix}$ and $B=\begin{pmatrix}
y & -u \\
-v & x
\end{pmatrix}$. Set $M=\coker A$. \\

Applying $\Hom_R(-,M)$ to (\ref{egTokuji}.1), we have
\begin{equation} \tag{\ref{egTokuji}.2}
0 \to M^2 \overset{^tA}{\longrightarrow} M^2 \overset{^tB}{\longrightarrow} M^2 \overset{^tA}{\longrightarrow} M^2 \overset{^tB}{\longrightarrow} \cdots.
\end{equation}
Since $^tB:M^2 \to M^2$ is injective, we conclude that $\Ext^{2i-1}_R(M,M)=0$ for all $i\geq 1$.

Let $\fp:=(x,y,u,v)$. Then $\fp$ is a prime ideal of $R$ of height three. Localizing (\ref{egTokuji}.1) at $\fp$, we get the following exact sequence:
\begin{equation}\tag{\ref{egTokuji}.3}
\cdots \overset{B}{\longrightarrow} R_\fp^2 \overset{A}{\longrightarrow} R_\fp^2 \overset{B}{\longrightarrow} R_\fp^2 \overset{A}{\longrightarrow} R_\fp^2 \to M_\fp \to 0.
\end{equation}
Since $x,y,u,v \in \fp R_\fp$, (\ref{egTokuji}.3) is a minimal free resolution of $M_\fp$.
In particular $M_\fp$ is not a free $R_\fp$-module, i.e., $M$ is not locally free on $X^3(R)$.
\end{eg}

\section*{Acknowledgments}
We are grateful to Lars Christensen, Mohsen Gheibi, Greg Piepmeyer, Srikanth Iyengar and Naoki Taniguchi for their feedback on the manuscript.

Part of this work was completed when Araya visited West Virginia University in January and February 2017. He is grateful for the kind hospitality of the WVU Department of Mathematics.


\begin{thebibliography}{10}
\bibitem{Ar}
Tokuji Araya, \emph{The {A}uslander-{R}eiten conjecture for {G}orenstein
  rings}, Proc. Amer. Math. Soc. \textbf{137} (2009), no.~6, 1941--1944.

\bibitem{AuRe}
Maurice Auslander and Idun Reiten, \emph{On a generalized version of the
  {N}akayama conjecture}, Proc. Amer. Math. Soc. \textbf{52} (1975), 69--74.

\bibitem{ABS}
Luchezar~L. Avramov, Ragnar-Olaf Buchweitz, and Liana~M. {\c{S}}ega,
  \emph{Extensions of a dualizing complex by its ring: commutative versions of
  a conjecture of {T}achikawa}, J. Pure Appl. Algebra \textbf{201} (2005),
  no.~1-3, 218--239. \MR{2158756}

\bibitem{BH}
Winfried Bruns and J\"urgen Herzog, \emph{Cohen-{M}acaulay rings}, Cambridge
  Studies in Advanced Mathematics, vol.~39, Cambridge University Press,
  Cambridge, 1993.

\bibitem{CeD2}
Olgur Celikbas and Hailong Dao, \emph{Necessary conditions for the depth
  formula over {C}ohen-{M}acaulay local rings}, J. Pure Appl. Algebra
  \textbf{218} (2014), 522--530.

\bibitem{CT}
Olgur Celikbas and Ryo Takahashi, \emph{Auslander-{R}eiten conjecture and
  {A}uslander-{R}eiten duality}, J. Algebra \textbf{382} (2013), 100--114.

\bibitem{CH1}
Lars~Winther Christensen and Henrik Holm, \emph{Algebras that satisfy
  {A}uslander's condition on vanishing of cohomology}, Math. Z. \textbf{265}
  (2010), no.~1, 21--40.

\bibitem{CH2}
\bysame, \emph{Vanishing of cohomology over {C}ohen-{M}acaulay rings},
  Manuscripta Math. \textbf{139} (2012), no.~3-4, 535--544.

\bibitem{GT}
Shiro Goto and Ryo Takahashi, \emph{On the {A}uslander-{R}eiten conjecture for
  {C}ohen-{M}acaulay local rings.}, Proc. Amer. Math. Soc. (to appear).

\bibitem{HL}
Craig Huneke and Graham~J. Leuschke, \emph{On a conjecture of {A}uslander and
  {R}eiten}, J. Algebra \textbf{275} (2004), no.~2, 781--790.

\bibitem{CSV}
Craig Huneke, Liana~M. {\c{S}}ega, and Adela~N. Vraciu, \emph{Vanishing of
  {E}xt and {T}or over some {C}ohen-{M}acaulay local rings}, Illinois J. Math.
  \textbf{48} (2004), no.~1, 295--317.

\bibitem{Nak}
Tadasi Nakayama, \emph{On algebras with complete homology}, Abh. Math. Sem.
  Univ. Hamburg \textbf{22} (1958), 300--307.

\bibitem{OnYos}
Maiko Ono and Yuji Yoshino, \emph{An {A}uslander--{R}eiten principle in derived
  categories}, J. Pure Appl. Algebra \textbf{221} (2017), no.~6, 1268--1278.

\bibitem{Roit}
Joseph J. Rotman, \emph{An introduction to homological algebra}, Universitext, Springer, New York, 2009.

\bibitem{Tach}
Hiroyuki Tachikawa, \emph{Quasi-{F}robenius rings and generalizations. {${\rm
  QF}-3$} and {${\rm QF}-1$} rings}, Lecture Notes in Mathematics, vol. 351,
  Springer-Verlag, Berlin-New York, 1973.

\bibitem{YKI}
Ken-ichi, Yoshida, \emph{Tensor products of perfect modules and maximal surjective
              {B}uchsbaum modules}, J. Pure Appl. Algebra, \textbf{123} (1998), no.~ 1-3, 313--326.


\bibitem{Yo}
Yuji Yoshino, \emph{Cohen-{M}acaulay modules over {C}ohen-{M}acaulay rings},
  London Mathematical Society Lecture Note Series, vol. 146, Cambridge
  University Press, Cambridge, 1990.

\end{thebibliography}
\end{document}